\documentclass[12pt]{amsart}
\newtheorem{theorem}{Theorem}

\newtheorem{corollary}[theorem]{Corollary}

\def\Re{\operatorname{Re}}

\title
{Exhausting domains of the symmetrized bidisc}
\author{Peter Pflug, W\l odzimierz Zwonek}

\address{Carl von Ossietzky Universit\"at Oldenburg\\
Institut f\"ur Mathe\-ma\-tik\\ Postfach 2503\\ D-26111 Oldenburg,
Germany}\email{pflug@mathematik.uni-oldenburg.de}

\address{Instytut Matematyki, Uniwersytet Jagiello\'nski, \L ojasiewicza 6,
30-348 Krak\'ow, Poland}\email{Wlodzimierz.Zwonek@im.uj.edu.pl}
\begin{document}



\begin{thanks}{This paper was written during the stay of the second
author at the Carl von Ossietzky Universit\"at Oldenburg,
(August 2010) supported by the DFG grant
No. 436POL113/103/0-2. The second author was also supported by the research
grant No. N N201 361436 of the Polish Ministry of Science and
Higher Education.}
\end{thanks}

\begin{abstract} We show that the symmetrized bidisc may be exhausted by strongly linearly convex domains. It shows in particular the existence of a strongly linearly convex domain that cannot be exhausted by domains biholomorphic to convex ones.
\end{abstract}

\maketitle
In our paper we show that the symmetrized bidisc can be exhausted by strongly linearly convex domains. Since the symmetrized bidisc is a $\Bbb C$-convex domain that cannot be exhausted by domains biholomorphic to convex ones, this fact has many interesting consequences. It gives a solution to open problems and implies alternate proofs of known results for the symmetrized bidisc.

Recall that a domain $D\subset\Bbb C^n$ is \textit{$\Bbb C$-convex} if for any complex line $\ell$ intersecting $D$ the intersection $\ell\cap D$ is connected and simply connected. A bounded domain $D\subset\Bbb C^n$ with $C^2$-boundary is called \textit{strongly linearly convex} if the defining function $r$ of $D$ satisfies the inequality
\begin{equation}
  \sum\sb{j,k=1}\sp{n}\frac{\partial^2r}{\partial z_j\bar\partial z_k}(z_0)X_j\bar X_k>\left|\sum\sb{j,k=1}\sp{n}\frac{\partial^2 r}{\partial z_j\partial z_k}(z_0)X_jX_k\right|
\end{equation}
for any boundary point $z_0$ and any non-zero vector $X$ from the complex tangent space to $\partial D$ at $z_0$.

Basic facts on $\Bbb C$-convex domains and strongly linearly convex ones that we use in the paper can be found in \cite{APS} and \cite{Hor}. Let us recall only that strong linear convexity implies $\Bbb C$-convexity.

For $\epsilon\in[0,1)$ let us define
\begin{equation}
D_{\epsilon}:=\{(s,p)\in\Bbb C^2:\sqrt{|s-\bar s p|^2+\epsilon}+|p|^2<1\}.
\end{equation}
Note that $D_0$ is the symmetrized bidisc $\Bbb G_2$ (see \cite{Agl-You} for the above description of the symmetrized bidisc) and $D_{\epsilon}\nearrow\Bbb G_2$ as $\epsilon\to 0^+$. Moreover, $\overline{D_{\epsilon}}\subset\Bbb C\times\Bbb D$, $\epsilon\in(0,1)$.

Note that the mapping
\begin{equation}
\Bbb C\times\Bbb D\owns (s,p)\mapsto (s-\bar sp,p)\in\Bbb C^2
\end{equation}
is an $\Bbb R$-diffeomorphism onto the image.
It shows in particular that $D_{\epsilon}$ is $\Bbb R$-diffeomorphic to the convex domain $G_{\epsilon}=\{(w,z)\in\Bbb C^2:\sqrt{|w|^2+\epsilon}+|z|^2<1\}$. Moreover, it is elementary to see that the strongly convex domains $G_{\epsilon}$ ($\epsilon\in(0,1)$) exhaust the (non-strongly) convex domain $G_0$.

We show that a similar result holds for the domains $D_{\epsilon}$.

\begin{theorem}\label{main} The domain $D_{\epsilon}$ is strongly linearly convex, $\epsilon\in(0,1)$. Consequently, the symmetrized bidisc can be exhausted by an increasing sequence of strongly linearly convex domains.
\end{theorem}

Combining Theorem \ref{main} with the fact that the symmetrized bidisc cannot be exhausted by domains biholomorphic to convex ones (see \cite{Edi}) we get the following corollary which gives a negative answer to a long-standing open problem on the existence of a strongly linearly convex domain not biholomorphic to a convex domain. Note that examples of strongly linearly convex domains which are not convex are well known (see \cite{STE} and also \cite{APS}).

\begin{corollary}\label{cor} The domains $D_{\epsilon}$ for $\epsilon>0$ small enough are examples of strongly linearly convex domains that are not biholomorphic to convex ones (and even cannot be exhausted by such domains).
\end{corollary}

\noindent{\bf Remark.} Recall that the equality between the Lempert function and the Carath\'eodory distance (i.e.~ the Lempert Theorem) holds for strong\-ly linearly convex domains (see \cite{Lem2}). Therefore, Theorem \ref{main} implies that the equality between the two functions on the symmetrized bidisc follows directly from the Lempert Theorem. It gives an alternate proof of that fact (to that in \cite{Agl-You} and \cite{Cos}). Moreover, it also implies that the tetrablock is the only known non-trivial  example of a domain (i.e.~ bounded and pseudoconvex) for which the fact that Lempert Theorem holds does not follow directly from the papers \cite{Lem} and \cite{Lem2} (see \cite{EKZ}).

\noindent{\bf Remark.} Theorem \ref{main} shows that the two papers of Lempert (see \cite{Lem} and \cite{Lem2}) verify the equality of the Lempert function and the Carath\'eodory distance for different classes of domains (convex ones and strongly linearly convex). This fact seemed to be unknown.

\noindent{\bf Remark.} Theorem \ref{main} also implies that $\Bbb G_2$ is a $\Bbb C$-convex domain - it gives an alternate proof to that in \cite{NPZ}.

Below we choose one of possible (global) defining $C^{\infty}$ functions for the domain $D_{\epsilon}$ ($\epsilon\in(0,1)$):
\begin{equation}
r_{\epsilon}(s,p):=r(s,p):=|s-\bar sp|^2+\epsilon-(1-|p|^2)^2,\;(s,p)\in\Bbb C\times\Bbb D.
\end{equation}

Note that the defining function is even real analytic.

\begin{proof}[Proof of Theorem \ref{main}] Let us fix $\epsilon\in(0,1)$.

First we note that the gradient of $r$ does not vanish on $\partial D_{\epsilon}$ (we shall calculate the complex tangent below).

Now for a point $(s_0,p_0)\in\partial D_{\epsilon}$ and $(s,p)$ being a non-zero tangent (in the complex sense) vector to $\partial D_{\epsilon}$, we shall show that $\rho_{\lambda\bar\lambda}(0)>|\rho_{\lambda\lambda}(0)|$,
where $\rho(\lambda):=r(s_0+\lambda s,p_0+\lambda p)$, $\lambda\in\Bbb C$. Note that for $\rho(s_0,p_0)=0$ and arbitrary $(s,p)$ we have

\begin{multline}
\rho(\lambda)=
2\Re\Big(\big((\bar s_0-s_0\bar p_0)(s-\bar s_0p)-(s_0-\bar s_0p_0)s\bar p_0+2\bar p_0p-2|p_0|^2\bar p_0p\big)\lambda\Big)+\\
|\lambda|^2\Big(|s-\bar s_0p|^2+|s|^2|p_0|^2-2\Re((\bar s_0-s_0\bar p_0)\bar s p)+2|p|^2-2|p_0|^2|p|^2\Big)\\
-\Re\big(2(s-\bar s_0 p)s\bar p_0\lambda^2\big)-\big(\Re(2\bar p_0 p\lambda)\big)^2+o(\lambda^2).
\end{multline}

The above formula shows in particular that tangent vectors $(s,p)$ to $\partial D_{\epsilon}$ are given by the formula
\begin{equation}
s(\bar s_0-s_0\bar p_0-\bar p_0(s_0-\bar s_0p_0))=p(\bar s_0(\bar s_0-s_0\bar p_0)-2\bar p_0+2|p_0|^2\bar p_0).
\end{equation}

It is also elementary to see that for a $C^2$-function $v(\lambda)=\Re(A\lambda)+a|\lambda|^2+\Re(b\lambda^2)-(\Re(c\lambda))^2+o(\lambda^2)$, where $a\in\Bbb R$, $A, b,c\in\Bbb C$, the condition for $v_{\lambda\bar\lambda}(0)>|v_{\lambda\lambda}(0)|$ is
\begin{equation}
a-\frac{|c|^2}{2}>\left|b-\frac{c^2}{2}\right|.
\end{equation}

Applying this information to the function $\rho$ we get the following inequality
\begin{multline}
|s-\bar s_0 p|^2+|s|^2|p_0|^2-2\Re((\bar s_0-s_0\bar p_0)\bar s p)+2|p|^2-2|p_0|^2|p|^2-\frac{|2\bar p_0p|^2}{2}>\\
\left|2(s-\bar s_0 p)s\bar p_0+\frac{(2\bar p_0p)^2}{2}\right|
\end{multline}
that when proven for boundary points $(s_0,p_0)$ and non-zero tangent $(s,p)$ will finish the proof of the theorem.

Substitute the condition on the tangency of the vector $(s,p)$. Since the inequality is trivial when $s_0=0$ we may neglect this case.

Then we divide both sides by $|p|^2$ and after reductions we get the inequality
\begin{multline}
\Big|2|p_0|^2\bar p_0-2\bar p_0+\bar s_0\bar p_0(s_0-\bar s_0p_0)\Big|^2+\\
|p_0|^2|\bar s_0(\bar s_0-s_0\bar p_0)-2\bar p_0+2|p_0|^2\bar p_0|^2-\\
2\Re\Big(((\bar s_0-s_0\bar p_0)(s_0(s_0-\bar s_0 p_0)-2p_0+2|p_0|^2p_0)(\bar s_0-s_0\bar p_0-\bar p_0(s_0-\bar s_0p_0))\Big)+\\
2|\bar s_0-s_0\bar p_0-\bar p_0(s_0-\bar s_0 p_0)|^2-4|p_0|^2|\bar s_0-s_0\bar p_0-\bar p_0(s_0-\bar s_0 p_0)|^2>\\
\Big|2(2|p_0|^2\bar p_0-2\bar p_0+\bar s_0\bar p_0(s_0-\bar s_0 p_0))(\bar s_0(\bar s_0-s_0\bar p_0)-2\bar p_0+2|p_0|^2\bar p_0)\bar p_0+\\
2\bar p_0^2(\bar s_0-s_0\bar p_0-\bar p_0(s_0-\bar s_0p_0))^2\Big|.
\end{multline}

Let us get rid of subscripts. After elementary calculations we get the inequality
\begin{multline}
|p|^2\Big|2|p|^2-2+\bar s(s-\bar s p)\Big|^2+|p|^2\Big|\bar s(s-s\bar p)+2|p|^2\bar p-2\bar p\Big|^2-\\
2\Re\Big((\bar s-s\bar p)(s(s-\bar s p)-2p+2|p|^2p)(\bar s-s\bar p-\bar p(s-\bar s p))\Big)+\\2|\bar s-s\bar p-\bar p(s-\bar s p)|^2-4|p|^2|\bar s-s\bar p-\bar p(s-\bar sp)|^2>\\
2|p|^2\Big|(2|p|^2-2+\bar s(s-\bar s p))(\bar s(\bar s-s\bar p)-2\bar p+2|p|^2\bar p)+\\
(\bar s-s\bar p-\bar p(s-\bar s p))^2\Big|.
\end{multline}

Note that the above function is invariant with respect to the mapping $(s,p)\mapsto(e^{it}s,e^{i2t}p)$ which means that we may assume that $s\geq 0$. Since $\rho(s,p)=0$ we get that $s^2=\frac{(1-|p|^2)^2-\epsilon}{|1-p|^2}$ (and $p$ may be arbitrary complex number satisfying the inequality $\epsilon\leq (1-|p|^2)^2$ ). Substituting the above in the inequality we get that

\begin{multline}
|p|^2\Big|2(|p|^2-1)(1-\bar p)+(1-|p|^2)^2-\epsilon\Big|^2+\\
|p|^2\Big|(1-|p|^2)^2-\epsilon-2\bar p(1-|p|^2)(1-p)\Big|^2-\\
2((1-|p|^2)^2-\epsilon)\cdot\\
\Re\left((1-\bar p)(\frac{(1-|p|^2)^2-\epsilon}{1-\bar p}-2p(1-|p|^2))(1-2\bar p+|p|^2)\right)+\\
2((1-|p|^2)^2-\epsilon)|1-2\bar p+|p|^2|^2-4|p|^2((1-|p|^2)^2-\epsilon)|1-2\bar p+|p|^2|^2>\\
2|p|^2\Big|(2(|p|^2-1)(1-\bar p)+(1-|p|^2)^2-\epsilon)((1-|p|^2)^2-\epsilon-2\bar p(1-|p|^2)(1-p))+\\
((1-|p|^2)^2-\epsilon)(1-2\bar p+|p|^2)^2\Big|,
\end{multline}
which is equivalent to the inequality
\begin{multline}
|1-2p+|p|^2|^22|p|^2\epsilon+2|p|^2\epsilon^2+\\
2\epsilon((1-|p|^2)^2-\epsilon)\Re(1-2p+|p|^2)>2|p|^2|\epsilon^2-\epsilon(1-2p+|p|^2)^2|.
\end{multline}
Note that $\Re(1-2p+|p|^2)=|1-p|^2>0$ which easily implies that the above inequality holds for all possible $p$ (i.e. satisfying the inequality $(1-|p|^2)^2\geq\epsilon$).

\end{proof}

\noindent{\bf Remark} Let us recall some of the open questions concerning the strongly linearly convex and $\Bbb C$-convex domains that still remain open and that can be found in \cite{APS} and \cite{ZZ}:

(a) Does the Lempert theorem hold for any bounded $\Bbb C$-convex domain?

(b) Can any bounded $\Bbb C$-convex domain be exhausted by strongly linearly convex ones ?  The answer is positive under an additional assumption of smoothness of $D$; see \cite{Jac}.

\end{document}